\documentclass{amsart}
\usepackage{amssymb, amsthm, amsmath, amsfonts}
\usepackage{graphics}
\usepackage{hyperref}
\usepackage[all]{xy}
\usepackage{enumerate}
\usepackage[mathscr]{eucal}
\usepackage{thmtools}

\theoremstyle{plain}
\newtheorem{theorem}{Theorem}[section]
\newtheorem{lemma}[theorem]{Lemma}

\newtheorem{corollary}[theorem]{Corollary}
\numberwithin{equation}{section}

\theoremstyle{definition}

\newtheorem{definition}[theorem]{Definition}
\newtheorem{example}[theorem]{Example}
\newtheorem{remark}[theorem]{Remark}

\newcommand{\C}{{\mathscr{C}}}

\newcommand{\T}{{\mathscr{T}}}
\newcommand{\FI}{{\mathscr{FI}}}
\newcommand{\OI}{{\mathscr{OI}}}
\newcommand{\VI}{{\mathscr{VI}}}
\newcommand{\FIM}{{\mathscr{FIM}}}
\newcommand{\m}{{\mathfrak{m}}}
\newcommand{\tC}{{\underline{\mathscr{C}}}}

\DeclareMathOperator{\Tor}{Tor}
\DeclareMathOperator{\gd}{gd}
\DeclareMathOperator{\Module}{-Mod}
\DeclareMathOperator{\module}{-mod}
\DeclareMathOperator{\lfMod}{-lfMod}
\DeclareMathOperator{\td}{td}
\DeclareMathOperator{\hd}{hd}
\DeclareMathOperator{\reg}{reg}
\DeclareMathOperator{\ann}{ann}

\title[An inductive machinery for representations of categories]{An inductive machinery for representations of categories with shift functors}

\author{Wee Liang Gan}
\address{Department of Mathematics, University of California, Riverside, CA 92521, USA.}
\email{wlgan@math.ucr.edu}

\author{Liping Li}
\address{Key Laboratory of High Performance Computing and Stochastic Information Processing (Ministry of Education), College of Mathematics and Computer Science, Hunan Normal University, Changsha, Hunan 410081, China.}
\email{lipingli@hunnu.edu.cn}

\thanks{The second author is supported by the National Natural Science Foundation of China 11541002, the Construct Program of the Key Discipline in Hunan Province, and the Start-Up Funds of Hunan Normal University 830122-0037. Both authors appreciate the anonymous referee for carefully reading the manuscript and providing quite many valuable and detailed comments, which highly improve the paper.}

\begin{document}

\begin{abstract}
We describe an inductive machinery to prove various properties of representations of a category equipped with a generic shift functor. Specifically, we show that if a property (P) of representations of the category behaves well under the generic shift functor, then all finitely generated representations of the category have the property (P). In this way, we obtain simple criteria for properties such as Noetherianity, finiteness of Castelnuovo-Mumford regularity, and polynomial growth of dimension to hold. This gives a systemetic and uniform proof of such properties for representations of the categories $\FI_G$ and $\OI_G$ which appear in representation stability theory.
\end{abstract}

\maketitle

\section{Introduction}

\subsection{Motivation}

The \emph{shift functor}, introduced by Church, Ellenberg, Farb, and Nagpal in \cite{CEF, CEFN}, has proven to be very useful in the representation theory of the category $\FI$, whose objects are the finite sets and morphisms are the injections between them. In particular, it plays a central role in the proofs of the following fundamental results: the category of finitely generated $\FI$-modules over a commutative Noetherian ring is abelian (\cite[Theorem A]{CEFN}); the \emph{Castelnuovo-Mumford regularity} of an $\FI$-module over any commutative ring has an upper bound given in terms of its first two homological degrees (\cite[Theorem A]{CE}); after applying the shift functor $S$ to a finitely generated $\FI$-module over a commutative Noetherian ring enough times, it becomes a $\sharp$-filtered module with particularly nice homological properties (see \cite[Theorem A]{N} and \cite[Theorem C]{LR}). The shift functor was also utilized by Ramos, Yu, and the authors to prove many other representation theoretic and homological properties of $\FI$-modules and $\FI_G$-modules in \cite{Gan1, Gan2, L1, L2, L3, LR, LY, R1, R2}. A key realization in these papers is that the shift functor allows one to make inductive arguments without relying too heavily on any special combinatorial structure of the categories $\FI$ and $\FI_G$.

Since many interesting combinatorial categories $\C$ appearing in representation stability theory are also equipped with shift functors (see \cite{GL}), it is natural to ask if one can use their shift functors to prove similar results. It would also be nice if one can axiomatize the properties of the shift functor used in the proofs. In this paper, we take a first step towards this by describing an inductive machinery for one to prove various properties of representations of categories with shift functors through a formal procedure. Roughly speaking, if a property (P) behaves well under the shift functor $S$, then all finitely generated representations of the category over a commutative Noetherian ring have the property (P). Moreover, to check whether (P) behaves well with $S$, it is often the case that one only needs to verify that the shifted modules of projective modules satisfy certain assumptions, which in practice are not hard. In particular, we provide some useful criteria for Noetherianity and finite Castelnuovo-Mumford regularity, and show that every finitely generated module of the category over a field must have the polynomial growth property.  These results immediately apply to the categories $\FI_G$ and $\OI_G$ (see \cite{L1, SS2} for definitions).

\subsection{Notation and terminologies}

To formulate the main result of this paper, we introduce some notation and terminologies. Throughout this paper, let $k$ be a commutative ring with identity and let $\mathbb{N}$ be the set of nonnegative integers. By $\tC$, we mean a category satisfying the following conditions:
\begin{itemize}
\item the objects of $\tC$ are parameterized by the nonnegative integers;
\item $\tC$ is \textit{$k$-linear}, i.e., $\tC$ is enriched over the category of $k$-modules; in particular, for every pair of objects $r, s$ in $\tC$, the set $\tC(r,s)$ of morphisms from $r$ to $s$ is a $k$-module;
\item $\tC$ is \textit{locally finite}, i.e., $\tC(r, s)$ is a finitely generated $k$-module for $r, s \in \mathbb{N}$;
\item $\tC$ is \textit{directed}, i.e., $\tC(r, s) = 0$ whenever $r > s$.
\end{itemize}

By definition, a \textit{representation} $V$ of $\tC$ (or a left \textit{$\tC$-module} $V$) is a covariant $k$-linear functor from $\tC$ to the category of $k$-modules. For any $\tC$-module $V$ and object $s$ of $\tC$, we shall write $V_s$ for $V(s)$. Where convenient, we shall also view $\tC$ as a $k$-algebra $\bigoplus_{0\leqslant r \leqslant s} \tC(r,s)$ and a $\tC$-module $V$ as $\bigoplus_{s\geqslant 0} V_s$. A representation $V$ is \textit{locally finite} if $V_s$ is a finitely generated $k$-module for every object $s$ of $\tC$. It is \textit{finitely generated} if $V$ contains a finite set such that the only $\tC$-submodule of $V$ containing this set is $V$ itself. For any $s \geqslant 0$, we say that $V$ is \emph{generated in degrees $\leqslant s$} if the only $\tC$-submodule of $V$ containing $\bigoplus_{r \leqslant s} V_r$ is $V$; we say $V$ is generated in degrees $\leqslant -1$ if $V = 0$. The \emph{generating degree} $\gd(V)$ of a $\tC$-module $V$ is defined to be the minimal $s\geqslant -1$ such that $V$ is generated in degrees $\leqslant s$; if no such $s$ exists, we set $\gd(V)$ to be $\infty$. (In some papers, the convention that the zero $\tC$-module has generating degree $-\infty$ is used. For us, it is more convenient to use the convention that the zero $\tC$-module has generating degree $-1$.)

The categories of all $\tC$-modules, locally finite $\tC$-modules and the category of finitely generated $\tC$-modules are denoted by $\tC \Module$, $\tC \lfMod$ and $\tC \module$ respectively. For any $s\geqslant 0$, define the $\tC$-module $M(s)$ to be the representable functor $\tC(s, -)$, so
\begin{equation*}
M(s)_r = \tC(s, r), \quad \mbox{ for each } r \geqslant 0;
\end{equation*}
equivalently, viewing $M(s)$ as a module over the $k$-algebra $\tC$, we have $M(s) = \tC e_s$ where $e_s$ is the identity element of the $k$-algebra $\tC(s, s)$. Then $M(s)$ is a projective $\tC$-module. The category of $\tC$-modules has enough projectives. It is plain that a $\tC$-module $V$ is finitely generated if and only if there is a surjection
\begin{equation*}
\bigoplus _{s \geqslant 0} M(s)^{\oplus m_s} \to V \quad \mbox{ where } \quad \sum_{s \geqslant 0} m_s < \infty.
\end{equation*}
In particular, $\tC \module$ is a full subcategory of $\tC \lfMod$. If $V\in \tC \lfMod$, then $V\in \tC \module$ if and only if $\gd(V) < \infty$.

For a nonnegative integer $n$, we say that a finitely generated $\tC$-module $V$ is \emph{$n$-finitely presented} if there is a projective resolution
\begin{equation*}
\ldots \to P^2 \to P^1 \to P^0 \to V \to 0
\end{equation*}
such that $P^i$ is finitely generated for each $i \leqslant n$; we say that $V$ is \textit{super finitely presented} if it is $n$-finitely presented for every $n \geqslant 0$.

\subsection{Generic shift functor}
We now introduce the definition of \textit{shift functor} which plays the central role in this paper.

\begin{definition}
A functor $\iota: \tC \to \tC$ is called a \emph{self-embedding functor} of degree 1 if it is faithful and $\iota(s) = s+1$ for every $s \geqslant 0$. A functor $S: \tC \Module \to \tC \Module$ is called a \emph{shift functor} if it is the pull-back functor $\iota^{\ast}$ induced by a self-embedding functor $\iota: \tC \to \tC$ of degree 1.
\end{definition}

Suppose $\tC$ has a shift functor $S$. It is clear that $S$ is an exact functor. It is also clear that $S$ preserves locally finite $\tC$-modules and hence restricts to a functor (which we still denote by $S$) from $\tC \lfMod$ to itself.

\begin{definition}
A shift functor $S: \tC \Module \to \tC \Module$ is called a \emph{generic} shift functor if there is a natural transformation $\mu : \mathrm{Id} \to S$ (where $\mathrm{Id}$ is the identity functor on $\tC \Module$) such that, for every $s\geqslant 0$, the map $\mu_{M(s)} : M(s) \to S M(s)$ is injective and its cokernel is generated in degrees $\leqslant s-1$.
\end{definition}

\emph{From now on, we suppose that $S$ is a generic shift functor as in the above definition.}

For every $\tC$-module $V$, we shall  write $KV$ and $DV$ for, respectively, the kernel and the cokernel of $\mu_V : V \to SV$. Thus, we have the following exact sequence
\begin{equation} \label{key sequence}
0 \to KV \to V \to SV \to DV \to 0.
\end{equation}
If $0\to U\to V\to W\to 0$ is a short exact sequence of $\tC$-modules, then by the snake lemma, there is an exact sequence
\begin{equation*}
0 \to KU \to KV \to KW \to DU \to DV \to DW \to 0.
\end{equation*}
In particular, the functor $D: V\mapsto DV$ is right exact; it is called the \emph{derivative functor} on $\tC \Module$.

\begin{remark} \normalfont
The shift functor $S$ for $\FI$-modules is generic and, moreover, has the property that $S$ and $D$ both preserve projective $\FI$-modules (see \cite{CEFN}). However, in our definitions above (and in the results that we shall prove later), we do not require $S$ or $D$ to preserve projective $\tC$-modules.
\end{remark}

In this paper, we shall study the properties of $\tC$-modules that behave well under the generic shift functor.

\begin{definition}
Suppose $\T$ is a subcategory of $\tC \Module$ and $F : \T \to \T$ is a functor. We say that a property (P) of some $\tC$-modules is:
\begin{itemize}
\item
\emph{$F$-invariant on $\T$} if, for every $V\in \T$:
\begin{equation*}
\mbox{$V$ has property (P)} \quad \Longrightarrow \quad \mbox{$FV$ has property (P);}
\end{equation*}

\item
\emph{$F$-dominant on $\T$} if, for every $V\in \T$:
\begin{equation*}
\mbox{$FV$ has property (P)} \quad \Longrightarrow \quad \mbox{$V$ has property (P);}
\end{equation*}

\item
\emph{$F$-predominant on $\T$} if, for every $V\in \T$:
\begin{equation*}
\mbox{$FV$ has property (P) and $KV=0$} \quad \Longrightarrow \quad  \mbox{$V$ has property (P).}
\end{equation*}
\end{itemize}
\end{definition}

Our usage of the term ``property" in the above definition may seem colloquial to some readers. One can formalize this as follows. By a \emph{property} (P) of some $\tC$-modules, we mean a function $\mathrm{P}$ from the set of $\tC$-modules to the 2-element set $\{\mathrm{Yes},\mathrm{No}\}$; we say that a $\tC$-module $V$ \emph{has property} (P) if $\mathrm{P}(V)=\mathrm{Yes}$.

\subsection{The inductive machinery}
Let us briefly describe the induction argument that we shall use in this paper.

We shall see later (in Lemma \ref{main technical lemma}) that if $V\in \tC \module$, then $SV, DV \in \tC \module$. Let (P) be a property of some $\tC$-modules and assume that the zero $\tC$-module has property (P). Suppose, for now, that the following three conditions hold:
\begin{itemize}
\item[{\bf (C1):}] (P) is $S$-dominant on $\tC \module$.

\item[{\bf (C2):}] (P) is $D$-predominant on $\tC \module$.

\item[{\bf (C3):}] In every short exact sequence $0\to V' \to V\to V'' \to 0$ of finitely generated $\tC$-modules, if $V'$ and $V''$ have property (P), then $V$ has property (P).
\end{itemize}
To show that every finitely generated $\tC$-module $V$ has property (P), we proceed by induction on the generating degree $\gd(V)$ by following the steps outlined below.

\medskip

\textit{Step 1:} Suppose, first, that $KV=0$. Since $\gd(DV) < \gd(V)$ by Statement (2) of Lemma \ref{main technical lemma}, by the induction hypothesis, $DV$ has property (P), and hence by condition (C2), $V$ has property (P).

\medskip

\textit{Step 2:} If $KV\neq 0$, then we construct an increasing chain
\begin{equation} \label{increasing chain}
0 = U^0 \subseteq U^1 \subseteq U^2 \subseteq U^3 \subseteq \ldots
\end{equation}
of $\tC$-submodules of $V$ such that:
\begin{itemize}
 \item $U^1$ is the kernel of $V\to SV$,
 \item $U^2/U^1$ is the kernel of $V/U^1 \to S(V/U^1)$,
 \item $U^3/U^2$ is the kernel of $V/U^2 \to S(V/U^2)$,
 \item and so on.
\end{itemize}
Let
\begin{equation*}
U = \bigcup_{n\geqslant 0} U^n.
\end{equation*}
We show that $K(V/U)=0$. Since $\gd(V/U)\leqslant \gd(V)$, it follows by the induction hypothesis or by Step 1 that $V/U$ has property (P).

\medskip

\textit{Step 3:} We prove that, under certain conditions, the increasing chain \eqref{increasing chain} must eventually stabilize, i.e. for some $n$, one has $U^n= U^{n+1} = U^{n+2} = \cdots$, and so $V/U^n = V/U$. Thus, from Step 2, we know that $V/U^n$ has property (P).

\medskip

\textit{Step 4:} For each $i\geqslant 1$, we have the short exact sequence
\begin{equation*}
0 \to V/U^i \to S(V/U^{i-1}) \to D(V/U^{i-1}) \to 0.
\end{equation*}
Note that $\gd(D(V/U^{i-1}))<\gd(V/U^{i-1})$ by Lemma \ref{main technical lemma} and $\gd(V/U^{i-1}) \leqslant \gd(V)$ since $V/U^{i-1}$ is a quotient module of $V$. By the induction hypothesis, we know that $D(V/U^{i-1})$ has the property (P). If $V/U^i$ has the property (P), then by condition (C3), so does $S(V/U^{i-1})$, and hence by condition (C1), so does $V/U^{i-1}$. Since $V/U^n$ has the property (P) by Step 3, we may do a downward induction to deduce that $V/U^i$ has property (P) for every $i\leqslant n$. In particular, $V=V/U^0$ has the property (P).

\subsection{Main results}

Our first main result generalizes \cite[Theorem A]{CEFN} by Church, Ellenberg, Farb, and Nagpal for $\FI$.

\begin{theorem}[Noetherianity] \label{noetherianity}
Suppose that $k$ is a commutative Noetherian ring, and $\tC$ has a generic shift functor $S$. Then every finitely generated $\tC$-module is Noetherian if and only if $S M(s)$ is finitely presented for every $s \geqslant 0$.
\end{theorem}

\begin{remark} \normalfont
It is well known that the following statements are equivalent:
\begin{enumerate}
\item every finitely generated module is super finitely presented;
\item every finitely generated module is finitely presented.
\item every finitely generated module is Noetherian.
\end{enumerate}
We prove Theorem \ref{noetherianity} by applying the above induction argument, taking (P) to be the finitely presented property.
\end{remark}

\begin{remark} \normalfont
Of course, categories equipped with generic shift functors are relatively rare. However, since it is well known that quotients of left Noetherian rings are still left Noetherian, we can deduce the following result: \emph{Let $\tC$ and $\tC'$ be two small $k$-linear categories and suppose that there is a full functor $\pi: \tC \to \tC'$ such that $\pi$ is a bijection on objects. If $\tC$ is locally Noetherian, so is $\tC'$.}

For example, let us consider the skeletal category $\C$ of $\FI$ whose objects are parameterized by nonnegative integers. Take a fixed integer $s \geqslant 3$. One can define a category $\C'$ which has the same objects as $\C$. Morphisms in $\C'$ are defined as follows: for two objects $r, t$, $\C'(r, t)$ is empty if $r > t$; is the same as $\C (r, t)$ if $t \geqslant r > s$ or $s > t \geqslant r$; and has only one element if $r \leqslant s \leqslant t$. Intuitively, $\C'$ is obtained from $\C$ by identifying those morphisms starting from an object less than or equal to $s$ and ending at another object greater than or equal to $s$.

Let $\tC$ and $\tC'$ be the $k$-linearizations of $\C$ and $\C'$ respectively, where $k$ is a commutative Noetherian ring. The category $\tC' \Module$ does not have a generic shift functor as the group $\C'(s-1, s-1) \cong \C(s-1, s-1)$ can not be embedded into the trivial group $\C'(s, s)$. However, there is a natural full functor $\pi: \tC \to \tC'$ which is the identity map on objects. Therefore, we can deduce the locally Noetherian property of $\tC'$ from the corresponding result of $\tC$.
\end{remark}

The second main result of this paper summarizes the inductive machinery in a form convenient for applications.

\begin{theorem}[Inductive machinery] \label{main theorem}
Suppose that $k$ is a commutative Noetherian ring and $\tC$ has a generic shift functor $S$ such that $S M(s)$ is finitely presented for every $s \geqslant 0$. Let (P) be a property of some $\tC$-modules and suppose that the zero $\tC$-module has property (P). Then every finitely generated $\tC$-module has property (P) if and only if conditions (C1), (C2), and (C3) hold.
\end{theorem}

By Theorem \ref{noetherianity}, the assumptions in Theorem \ref{main theorem} imply that every finitely generated $\tC$-module is Noetherian. This ensures that we can carry out Step 3 in our induction argument.

\subsection{Applications}

The inductive machinery allows one to prove interesting properties (P) by checking that it satisfies the three conditions specified in Theorem \ref{main theorem}. To illustrate its usefulness, we need to recall a few more definitions.

Define a two-sided ideal $\m$ of $\tC$ (viewed as a $k$-algebra) by:
\begin{equation*}
\m = \bigoplus _{0 \leqslant r < s} \tC (r, s).
\end{equation*}
Let us view a $\tC$-module $V$ as a $k$-module with the direct sum decomposition:
\begin{equation*}
V = \bigoplus _{s \geqslant 0} V_s.
\end{equation*}
Note that $V$ is generated in degrees $\leqslant s$ (or equivalently, there is an epimorphism $P \to V$ where $P$ is a direct sum of $\tC$-modules of the form $M(r)$ with $r \leqslant s$) if and only if the value of $V / \m V$ is 0 on each object $r>s$, so $V / \mathfrak{m} V$ encapsulates information about generators of $V$. One defines the \emph{0-th homology group} by:
\begin{equation*}
H_0(V) = V/\m V \cong \tC/\m \otimes _{\tC} V.
\end{equation*}
Correspondingly, for $i \geqslant 1$, one defines the \emph{$i$-th homology group} using the $i$-th left derived functor of $H_0$, i.e.
\begin{equation*}
H_i(V) = \Tor ^{\tC} _i (\tC/\m, V).
\end{equation*}
Note that these homology groups inherit a natural $\tC$-module structure from $V$. For any $i \geqslant 0$, we define the \emph{$i$-th homological degree} of $V$ to be
\begin{equation*}
\hd_i (V) = \sup \{ s \geqslant 0 \mid \text{ the value } (H_i(V))_s \neq 0 \};
\end{equation*}
we set it to be $-1$ whenever the above set is empty. Note that the zeroth homological degree $\hd_0(V)$ is the generating degree of $V$, i.e. $\hd_0(V) = \gd(V)$. We say that $V$ is \emph{$n$-presented in finite degrees} if $\hd_i(V) < \infty$ for every $i \leqslant n$. We define the \emph{Castelnuovo-Mumford regularity} of $V$ to be
\begin{equation*}
\reg(V) = \sup \{ \hd_i(V) - i \mid i \geqslant 0 \}.
\end{equation*}

\begin{remark} \normalfont
In the literature, different conventions have been used. For instance, in some papers the homological degree is set to be $-\infty$ when the corresponded homology group is trivial. Also, in \cite{CE, L2, LR}, to define the Castelnuovo-Mumford regularity, the index $i$ is required to be strictly positive. Therefore, one has $\reg(M(s)) = s$ in the current setting, while according to the definition in some other papers, the regularity of $M(s)$ is $-\infty$. The reason we choose a different convention here is to simplify the statements of many results and their proofs.
\end{remark}

\begin{remark} \normalfont
The homological degrees of a $\tC$-module are closely related to properties such as finitely generated, finitely presented, etc. For instance, when $k$ is Noetherian, one easily observes that a locally finite $\tC$-module $V$ is $n$-finitely presented if and only if it is $n$-presented in finite degrees. In particular, if $V$ has finite Castelnuovo-Mumford regularity, then it is $n$-finitely presented for every $n \geqslant 0$.
\end{remark}

We now describe a few applications of Theorem \ref{main theorem}.

\subsubsection{Castelnuovo-Mumford regularity.}
The following corollary will be proved by taking (P) in Theorem \ref{main theorem} to be the property of having finite Castelnuovo-Mumford regularity.

\begin{corollary} \label{cm regularity}
Suppose that $k$ is a commutative Noetherian ring and $\tC$ has a generic shift functor $S$. If $\reg (S M(s)) \leqslant s$ for all $s \geqslant 0$, then $\reg(V) < \infty$ for every finitely generated $\tC$-module $V$.
\end{corollary}

The proof of Corollary \ref{cm regularity} will be given in Section \ref{CM regularity section}.

\begin{remark} \normalfont
In \cite[Theorem A]{CE}, Church and Ellenberg proved a stronger result for $\FI$-modules by giving an upper bound for the regularity. We are not able to give such an upper bound in Corollary \ref{cm regularity}. The main reason is that for an $\FI$-module $V$, the kernel $KV$ of the natural map $V \to SV$ has a simple description and one can give an upper bound for $\reg(KV)$ (see, for example, \cite[Corollary 1.16]{L1}), but we do not have a similar simple description of $KV$ in the general case.
\end{remark}

\subsubsection{Polynomial growth.}
The following corollary extends (\cite[Theorem C]{CEFN}) to other categories.

\begin{corollary} \label{polynomial growth}
Suppose that $k$ is a field and $\tC$ has a generic shift functor $S$ such that $S M(s)$ is finitely presented for every $s \geqslant 0$. Then for any finitely generated $\tC$-module $V$, the Hilbert function
\begin{equation*}
H_V: \mathbb{N} \to \mathbb{N}, \quad n \mapsto \dim_k V_n
\end{equation*}
coincides with a rational polynomial with degree not exceeding $\gd(V)$ when $n \gg 0$.
\end{corollary}
\begin{proof}[Sketch of proof]
For any $\tC$-module $V$, let (P) be the property that there exists a polynomial $Q(t)\in \mathbb{Q}[t]$ with $\deg(Q)\leqslant \gd(V)$ such that $\dim_k V_n = Q(n)$ for all $n \gg 0$. It is easy to see that (P) is $S$-dominant on $\tC\module$ and $D$-predominant on $\tC\module$. Moreover, if two terms in a short exact sequence have this property, so does the third term. Hence, by Theorem \ref{main theorem}, every finitely generated $\tC$-module has the property (P).
\end{proof}

\subsubsection{$\sharp$-filtered $\FI$-modules.} Let (P) be the following property of finitely generated $\FI$-modules $V$ over a commutative Noetherian ring: $S^n V$ is a $\sharp$-filtered module (see \cite{N} for a definition) for $n \gg 0$. Clearly (P) satisfies conditions (C1) and (C3). Moreover, Lemma 3.12 in \cite{LY} asserts that property (P) satisfies condition (C2). Consequently, we deduce that \emph{every finitely generated $\FI$-module over a commutative Noetherian ring has (P), a result first proved by Nagpal \cite[Theorem A]{N}.}

Let us mention here a property which does not fulfill the conditions in Theorem \ref{main theorem}. For $\FI$-modules over a commutative ring, let (P) be the property of having finite projective dimension. Clearly, (P) satisfies condition (C3). In \cite{LY}, Li and Yu showed that it is $D$-predominant, and is $S$-invariant. But the finite projective dimension of $SV$ does not imply the finite projective dimension of $V$, so (P) is not $S$-dominant. It is well known that there are finitely generated $\FI$-modules with infinite projective dimension.

\subsection{$\OI_G$-modules} Let $G$ be a finite group. The category $\C = \OI_G$ has objects the nonnegative integers. For a pair of objects $r, s \geqslant 0$, $\C (r, s)$ is the set of pairs $(f, g)$ where $f: \{1, \ldots, r \} \to \{1, \ldots, s \}$ is a strictly increasing map and $g: \{ 1, \ldots, r\} \to G$ is an arbitrary map. For $(f_1, g_1) \in \C(r, s)$ and $(f_2, g_2)\in \C (s, t)$, their composition is $(f_3, g_3)$ where
\begin{equation*}
f_3 = f_2 \circ f_1, \quad \text{and } g_3 (i)= g_2 (f_1(i)) \cdot g_1(i)
\end{equation*}
for $1 \leqslant i \leqslant r$. Let $\tC$ be the $k$-linearization of $\C$.

In \cite{GL}, we constructed a generic shift functor $S$ for $\tC$-modules with the property that:
\begin{equation*}
S M(s) \cong M(s) \oplus M(s-1)^{\oplus |G|}, \quad \mbox{ for each } s \geqslant 0.
\end{equation*}
Consequently, one immediately deduces the following results first proved by Sam and Snowden in \cite{SS2}: \emph{every finitely generated $\OI_G$-module over a commutative Noetherian ring is Noetherian, and the Hilbert function of any finitely generated $\OI_G$-module over  a field is eventually polynomial with rational coefficients.}

The following result, to the best of our knowledge, is new:

\begin{corollary}
Suppose that $k$ is a commutative Noetherian ring. If $V$ is a finitely generated $\OI_G$-module, then $\reg(V)<\infty$.
\end{corollary}
\begin{proof}
Immediate from Corollary \ref{cm regularity}.
\end{proof}

\subsection{Further problems}
If $V$ is a finitely generated $\OI_G$-module, is $S^n V$ a $\sharp$-filtered module (defined in a similar way as $\FI_G$-modules) when $n \gg 0$? To prove this, it suffices to show that the property of being $\sharp$-filtered after applying $S$ enough times is $D$-predominant since conditions (C1) and (C3) hold obviously. Unfortunately, at this moment we do not know whether condition (C2) holds. Actually, the proof of this fact for $\FI_G$-modules uses the following key observation: $S$ and $D$ commute for $\FI_G$-modules. However, this does not hold for $\OI_G$-modules; see Remark \ref{example} for more details.

In \cite{G, SS2}, Gadish, Sam and Snowden considered some combinatorial categories closely related to $\FI$, including $\FI^r$, $\FI_d$, $\FIM$, $\VI$, etc. Our methodology cannot apply to all of them immediately. For instance, if $\tC$ is the $k$-linearization of $\FI_d$ with $d > 1$, the Hilbert functions of $M(s)$ for $s \geqslant 1$ may not be eventually polynomial, so there is no \emph{generic} shift functor for the category of $\FI_d$-modules (cf. Corollary \ref{polynomial growth}). However, in \cite{GL}, we have shown that $\FI_d$ have a natural self-embedding functor which induces a shift functor $S$ such that $S M(s) \cong M(s)^{\oplus d} \oplus M(s-1)^{\oplus s}$. From this observation, for a general $\FI_d$-module $V$, one may construct an exact sequence
\begin{equation*}
0 \to KV \to V^{\oplus d} \to SV \to DV \to 0
\end{equation*}
with $\gd(DV) < \gd(V)$. One can also construct a similar exact sequence for representations of $\FIM$. Based on these observations, it is natural to ask if a suitable modification of our inductive machinery can be applied to these examples.

\subsection{An inductive machinery without shift functor}

The main reason motivating us to focus on shift functors is that these functors are well known and have played a central role in representation stability theory. However, for readers interested in general representation theory, as the referee kindly pointed out, the inductive strategy described in this paper does not depend on the existence of a shift functor; instead, it works for general functors sharing similar properties with shift functors.

Explicitly, suppose that there is a functor $F: \tC \Module \to \tC \Module$ satisfying the following axioms:
\begin{itemize}
\item there is a natural transformation $\mu: F \to \mathrm{Id}$, where $\mathrm{Id}$ is the identity functor on $\tC \Module$;
\item for each $s \geqslant 0$, the natural map $\mu_{M(s)}: M(s) \to FM(s)$ is injective;
\item there is a positive number $N$ such that $\gd(V) - N \leqslant \gd(FV/ \mu_V(V)) < \gd(V)$ for all $\tC$-modules $V$.
\end{itemize}
Then with small modifications, one can prove similar versions of Theorems \ref{noetherianity} and \ref{main theorem}, and Corollary \ref{polynomial growth}. To show Corollary \ref{cm regularity}, one can impose the following two extra conditions:
\begin{itemize}
\item $\gd(V) - 1 = \gd(FV/ \mu_V(V))$ for all nonzero $\tC$-modules $V$;
\item $\reg(FM(s)) \leqslant s$.
\end{itemize}

\section{Preliminaries}

Suppose that $\tC$ has a generic shift functor $S$. Recall that for a $\tC$-module $V$ we have an exact sequence:
\begin{equation*}
0 \to KV \to V \to SV \to DV \to 0.
\end{equation*}
Moreover, a $\tC$-module homomorphism $\varphi: V \to W$ induces the following commutative diagram
\begin{equation*}
\xymatrix{
0 \ar[r] & KV \ar[r] \ar[d]^{K \varphi} & V \ar[r] \ar[d]^{\varphi} & SV \ar[r] \ar[d]^{S\varphi} & DV \ar[r] \ar[d]^{D \varphi} & 0 \\
0 \ar[r] & KW \ar[r] & W \ar[r] & SW \ar[r] & DW \ar[r] & 0.
}
\end{equation*}

We will need the following notion.

\begin{definition} \label{adaptable projective resolution}
A projective resolution $\ldots \to P^2 \to P^1 \to P^0 \to V \to 0$ is \emph{adaptable} if the following two conditions are satisfied:
\begin{itemize}
\item each $P^i$ is of the form $\bigoplus_{s\geqslant 0} M(s)^{\oplus m_s}$ where $0\leqslant m_s \leqslant \infty$;
\item $\gd(P^i) = \gd(Z^i)$ for $i \geqslant 0$, where $Z^0=V$, $Z^1$ is the kernel of $P^0 \to V$, and $Z^i$ is the kernel of $P^{i-1} \to P^{i-2}$ for $i \geqslant 2$.
\end{itemize}
\end{definition}

In the next lemma, we collect some important preliminary results which will be used intensively.

\begin{lemma} \label{main technical lemma}
Let $S$ be a generic shift functor and $V$ be a $\tC$-module. Let (P) be the $n$-presented in finite degrees property. One has:
\begin{enumerate}
\item If $U$ is a submodule of $V$ and the map $\mu_V: V \to SV$ is injective, then the map $\mu_U: U \to SU$ is injective as well.
\item $\gd(DV) = \gd(V) - 1$ whenever $V \neq 0$.
\item $\gd(SV) \leqslant \gd(V) \leqslant \gd(SV) + 1$.
\item Let $0 \to U \to V \to W \to 0$ be a short exact sequence of $\tC$-modules and suppose that $KW = 0$. Then it induces a short exact sequence $0 \to DU \to DV \to DW \to 0$.
\item Suppose that $SM(s)$ has (P) for every $s \geqslant 0$. Then (P) is $S$-dominant and $S$-invariant.
\item Suppose that $SM(s)$ has (P) for every $s \geqslant 0$. Then (P) is $D$-predominant.
\end{enumerate}
\end{lemma}

\begin{proof}
Since all conclusion holds for $V = 0$ trivially, we suppose that $V$ is nonzero.

(1): The inclusion $U \to V$ gives rise to the following commutative diagram by the naturality of $S$:
\begin{equation*}
\xymatrix{
U \ar[r] \ar[d] & V \ar[d]\\
SU \ar[r] & SV,
}
\end{equation*}
and the first statement follows immediately.

(2): This is a generalization of \cite[Lemma 1.5]{L3}. Firstly, for $s \geqslant 0$, $\gd (D M(s)) < s$ by definition of generic shift functors. But one also sees that $\gd(D M(s)) \geqslant s-1$ for $s \geqslant 1$ since the value of $D M(s)$ on object $s - 1$ is nonzero, and its values on all objects less than $s - 1$ (if such objects exist) are 0. The conclusion also holds for $M(0)$ by our convention that the generating degree of the zero module is $-1$.

Now consider a surjection $P \to V \to 0$, where $P$ is direct sum of those $M(s)$, $s \geqslant 0$, and $\gd(P) = \gd(V)$. Applying $D$ (which is right exact) one gets a surjection $DP \to DV \to 0$, so
\begin{equation*}
\gd(DV) \leqslant \gd(DP) = \gd(P) - 1 = \gd(V) - 1.
\end{equation*}
To show $\gd(DV) \geqslant \gd(V) - 1$, one just copies the arguments in \cite[Proposition 2.4]{LY} and \cite[Corollary 1.5]{L3}. Note that a different convention was used there; i.e., the generating degree of the zero module was set to be $-\infty$.

(3): Break the exact sequence $0 \to KV \to V \to SV \to DV \to 0$ to two short exact sequences $0 \to KV \to V \to V^{(1)} \to 0$ and $0 \to V^{(1)} \to SV \to DV \to 0$. From the first exact sequence we deduce that $\gd(V^{(1)}) \leqslant \gd(V)$, and from the second one we have
\begin{equation*}
\gd(SV) \leqslant \max \{ \gd(V^{(1)}), \, \gd(DV) \} \leqslant \gd(V).
\end{equation*}
On the other hand, the surjection $SV \to DV$ implies that $\gd(SV) \geqslant \gd(DV)$. But we have proved $\gd(DV) = \gd(V) - 1$. Thus $\gd(SV) \geqslant \gd(V) - 1$.

(4): The given short exact sequence gives rise to the following commutative diagram
\begin{equation*}
\xymatrix{
0 \ar[r] & U \ar[r] \ar[d]^{\mu_U} & V \ar[r] \ar[d]^{\mu_V} & W \ar[r] \ar[d]^{\mu_W} & 0\\
0 \ar[r] & SU \ar[r] & SV \ar[r] & SW \ar[r] & 0.
}
\end{equation*}
By the snake Lemma, we obtain an exact sequence
\begin{equation*}
0 \to KU \to KV \to KW = 0 \to DU \to DV \to DW \to 0.
\end{equation*}

(5): We use induction on $n$. The conclusion has been established for $n = 0$ in the previous statement. Let $i\geqslant 1$ and suppose that the conclusion is true for all numbers $n \leqslant i$. Now consider $n=i+1$. Clearly,
\begin{equation*}
\hd_j(V) < \infty, \, \forall \, 0 \leqslant j \leqslant i+1 \Leftrightarrow \hd_{i+1}(V) < \infty \text{ and } \hd_j(V) < \infty, \, \forall \, 0 \leqslant j \leqslant i.
\end{equation*}
By the induction hypothesis,
\begin{equation*}
\hd_{i+1}(V) < \infty \text{ and } \hd_j(V) < \infty, \, \forall \, 0 \leqslant j \leqslant i \Leftrightarrow \hd_{i+1}(V) < \infty \text{ and } \hd_j(SV) < \infty, \, \forall \, 0 \leqslant j \leqslant i.
\end{equation*}
Therefore, we may assume that $\hd_j(V)$ and $\hd_j(SV)$ are finite for $0 \leqslant j \leqslant i$, and show that $\hd_{i+1}(V) < \infty$ if and only if $\hd_{i+1} (SV) < \infty$.

Take a short exact sequence $0 \to W \to P \to V \to 0$ such that $P$ is projective and $\gd(P) = \gd(V)$. From the short exact sequence $0 \to SW \to SP \to SV \to 0$ we derive a long exact sequence of homology groups
\begin{equation*}
\ldots \to H_{i+1} (SP) \to H_{i+1}(SV) \to H_i(SW) \to H_i(SP) \to \ldots.
\end{equation*}
Since both $\hd_i(SP)$ and $\hd_{i+1} (SP)$ are finite, consequently, $\hd_{i+1}(SV) < \infty$ if and only if $\hd_i(SW) < \infty$, and by the induction hypothesis (replacing $V$ by $W$), if and only if $\hd_i(W) < \infty$. But $\hd_i(W) = \hd_{i+1}(V)$. The conclusion follows by induction.

(6): Suppose that the map $V \to SV$ is injective, and $\hd_i(DV) < \infty$ for $0 \leqslant i \leqslant n$. We want to show that $\hd_i(V) < \infty$ as well. Take an adaptable projective resolution $P^{\bullet} \to V \to 0$. The short exact sequence $0 \to V \to SV \to DV \to 0$ induces the following commutative, exact diagram by Statement (4) of this lemma:
\begin{equation*}
\xymatrix{
0 \ar[r] & P^{\bullet} \ar[r] \ar[d] & SP^{\bullet} \ar[r] \ar[d] & DP^{\bullet} \ar[r] \ar[d] & 0 \\
0 \ar[r] & V \ar[r] \ar[d] & SV \ar[r] \ar[d] & DV \ar[r] \ar[d] & 0\\
 & 0 & 0 & 0 &
}
\end{equation*}
The exact complex $DP^{\bullet} \to DV \to 0$ might not be a projective resolution since $DP^i$ in general is not projective. However, by Statement (2) of Lemma \ref{main technical lemma}, the equality $\gd(DZ^i) = \gd(DP^i)$ still holds for $i \geqslant 0$, where $Z^i$ is defined in Definition \ref{adaptable projective resolution}.

Clearly, the conclusion follows if we can prove that $\gd(P^i) < \infty$ for $0 \leqslant i \leqslant n$. By Statement (2), it suffices to show that $\gd(DP^i) < \infty$ for $0 \leqslant i \leqslant n$. But this is clear. Indeed, breaking the exact complex $DP^{\bullet} \to DV \to 0$ into short exact sequences, the first piece $0 \to DZ^1 \to DP^0 \to DV \to 0$ gives a long exact sequence of homology groups
\begin{equation*}
\ldots \to H_1(DV) \to H_0(DZ^1) \to H_0(DP^0) \to H_0(DV) \to 0.
\end{equation*}
Since $\gd(DV) < \infty$, we know that $\gd(V)$ and hence $\gd(P^0) < \infty$. Therefore, by the given assumption, both $\hd_i(DV)$ and $\hd_i(SP^0)$ are finite for $0 \leqslant i \leqslant n$. Moreover, by the short exact sequence $0 \to P^0 \to SP^0 \to DP^0 \to 0$ we conclude that $\hd_i(DP^0) < \infty$ for $0 \leqslant i \leqslant n$. Consequently, from the above long exact sequence we deduce that $\hd_i(DZ^1) < \infty$ for $0 \leqslant i \leqslant n-1$, and in particular $\gd(DP^1) = \gd(DZ^1) < \infty$. Similarly, from the second piece $0 \to DZ^2 \to DP^1 \to DZ^1 \to 0$ one deduces that $\hd_i(DZ^2) < \infty$ for $0 \leqslant i \leqslant n-2$, and in particular $\gd(DP^2) = \gd(DZ^2) < \infty$. Continuing this process we get the required conclusion.
\end{proof}

\begin{remark} \normalfont
For $\FI$-modules, these properties have been established and applied in literature; see for instance \cite{CEFN, CE, L2, L3, LR, LY, R1, R2}.
\end{remark}

If the category $\tC$ is locally finite, we deduce the following corollary.

\begin{corollary} \label{n finitely presented}
Let $k$ be a commutative Noetherian ring and suppose that $S M(s)$ is $n$-finitely presented for every $s\geqslant 0$. Then the $n$-finitely presented property is $S$-dominant on $\tC\lfMod$ and $D$-predominant on $\tC\lfMod$. Moreover, if $0 \to U \to V \to W \to 0$ is a short exact sequence in $\tC\lfMod$ such that $U$ and $W$ are $n$-finitely presented, so is $V$.
\end{corollary}

\begin{proof}
Since $k$ is Noetherian and $\tC(s, s)$ is a finitely generated $k$-module, it is also a left Noetherian ring for $s \geqslant 0$. Consequently, the category $\tC \lfMod$ of locally finite $\tC$-modules is abelian; that is, submodules of locally finite $\tC$-modules are still locally finite. Moreover, since $M(s)$ is locally finite for $s \geqslant 0$, this category has enough projectives.

Note that a module $V \in \tC \lfMod$ is finitely generated if and only if $\gd(V) < \infty$. Consequently, $V$ is $n$-finitely presented if and only if it is $n$-presented in finite degrees. The conclusion of the first statement follows immediately from the previous lemma.

To prove the second statement, one notes that the given short exact sequence induces a long exact sequence of homology groups
\begin{equation*}
\ldots \to H_{i+1}(W) \to H_i(U) \to H_i(V) \to H_i(W) \to \ldots
\end{equation*}
Therefore, if all $\hd_i(U)$ and $\hd_i(W)$ are finite for $0 \leqslant i \leqslant n$, then $\hd_i(V) < \infty$ for $0 \leqslant i \leqslant n$. The conclusion follows immediately.
\end{proof}

\section{A recursive construction}

Let (P) be a property of some $\tC$-modules. If the map $\mu_V: V \to SV$ is injective, then to show that $V$ has property (P), it suffices to prove that (P) is $D$-predominant, and $DV$ satisfies (P). However, in general, the map $V \to SV$ might have a nonzero kernel $KV$. We have the exact sequence
\begin{equation*}
0 \to V/KV \to SV \to DV \to 0.
\end{equation*}
Now, if (P) satisfies conditions (C1) and (C3), and if $V/KV$ and $DV$ both have (P), then so does $V$. We now replace $V$ by $V/KV$, and ask if $\mu_{V/KV} : V/KV \to S(V/KV)$ is injective. Repeating, we obtain a sequence $V/U^n$ of $\tC$-modules where $0=U^0 \subset U^1 \subset U^2 \subset \cdots$ is an increasing chain of $\tC$-submodules of $V$ such that
$U^n / U^{n-1} = K(V/U^{n-1})$.

\begin{lemma} \label{commutative diagram}
Let $V$ be a $\tC$-module and $n\geqslant 0$. Then $(U^{n+1})_s = \{ v \in V_s \mid \mu_V (v) \in (SU^n)_s \}$ for each object $s$ of $\tC$.
\end{lemma}
\begin{proof}
This is immediate from the following commutative diagram with exact rows:
\begin{equation*}
\xymatrix{
0 \ar[r] & U^n \ar[r] \ar[d] & V \ar[r] \ar[d]&  V/U^n \ar[r] \ar[d] & 0 \\
0 \ar[r] & SU^n \ar[r]  & SV \ar[r] & S(V/U^n) \ar[r] & 0 },
\end{equation*}
where we identify $SU^n$ with its image in $SV$ to deduce that $SV/SU^n \cong S(V/U^n)$ as $S$ is an exact functor.
\end{proof}

We let
\begin{equation*}
V_{\sin} = \bigcup_{n} U^n \quad \mbox{ and } \quad V_{\reg} = V/ V_{\sin},
\end{equation*}
where ``sin" means singular and ``reg" means regular.

\begin{remark}
One has (by \cite[Corollary 5.38]{Ro}):
\[ V_{\reg} = \varinjlim V/U^n.  \]
\end{remark}

\begin{lemma} \label{singular regular decomposition}
Let $V$ be a $\tC$-module. Then $K(V_{\reg})=0$.
\end{lemma}
\begin{proof}
Suppose $\bar{v} \in K(V_{\reg})_s \subset V_s/(V_{\sin})_s$. Choose a representative $v$ of $\bar{v}$ in $V_s$. Then $\mu_V(v) \in (S V_{\sin})_s$. In particular, there exists $n$ such that $\mu_V(v) \in (SU^n)_s$. Thus, $v \in U^{n+1}_s \subset (V_{\sin})_s$. Hence, $\bar{v}=0$.
\end{proof}

We use $\FI$-modules and $\OI$-modules to illustrate the above construction.

\begin{example} \normalfont
Let $\tC$ be the $k$-linearization of (a skeleton of) $\FI$, and let $S$ be the shift functor introduced in \cite{CEFN}. By \cite[Lemma 2.1]{L3},
\begin{equation*}
U^1 = KV = \bigoplus_{s \geqslant 0} \{ v \in V_s \mid \alpha \cdot v = 0, \, \forall \alpha \in \C(s, s+1) \};
\end{equation*}
and
\begin{equation*}
U^2/U^1 = K(V/KV) = \bigoplus _{s \geqslant 0} \{\bar{v} \in (V/KV)_s \mid \alpha \cdot \bar{v} = 0, \, \forall \alpha \in \C(s, s+1) \}.
\end{equation*}
Therefore,
\begin{equation*}
U^2 = \bigoplus_{n \geqslant 0} \{ v \in V_s \mid \alpha \cdot v = 0, \, \forall \alpha \in \C(s, s+2) \}.
\end{equation*}
Recursively, for $n \geqslant 1$, we conclude that
\begin{equation*}
U^n = \bigoplus_{s \geqslant 0} \{ v \in V_s \mid \alpha \cdot v = 0, \, \forall \alpha \in \C(s, s+n) \}.
\end{equation*}
Consequently, $V_{\sin}$ is precisely the torsion part of $V$, generated by elements $v \in V_s$, $s \geqslant 0$, such that there exist an object $r > s$ and a morphism $\alpha \in \C(s, r)$ with $\alpha \cdot v = 0$. Correspondingly, $V_{\reg}$ is the torsion-free part. For details, the reader can also refer to \cite[Remark 1.1]{LY}.

One can also describe these kernels more conceptually. Let $\mathfrak{m}$ be the two-sided ideal of $\tC$ (viewed as a $k$-algebra) consisting of finite combinations of morphisms between different objects in $\tC$. Then for $n \geqslant 1$
\begin{equation*}
U^n = \ann(\mathfrak{m}^n) = \{ v \in V \mid \mathfrak{m}^n v  = 0\}.
\end{equation*}
\end{example}

\begin{example} \normalfont
Let $\tC$ be the $k$-linearization of $\OI$, a subcategory of $\FI$ which has the same objects as $\FI$ and whose morphisms are injections preserving the natural order on $[s]$, $s \geqslant 0$. For $s \geqslant 1$, let $\iota_s: [s] \to [s+1]$ be the map by sending $i \in [s] \to i+1 \in [s+1]$. This family of maps gives rise to a faithful self-embedding functor, and induces a generic shift functor $S$; see \cite[Proposition 5.2 and Example 5.4]{GL}. The reader can check that for $s \geqslant 0$, $S M(s) \cong M(s) \oplus M(s-1)$.

Let $I$ be the two-sided ideal of $\tC$ (viewed as a $k$-algebra) generated by $\{ \iota_s \mid s \geqslant 0 \}$. The reader can check that $I$ as a free $k$-module is spanned by all morphisms $\alpha: [r] \to [s]$ in $\OI$ such that $\alpha(1) \neq 1$, where $s > r \geqslant 1$. For an $\OI$-module $V$, one has
\begin{equation*}
KV = \bigoplus _{s \geqslant 0} \{v \in V_s \mid \iota_s \cdot v  = 0 \} = \ann (I).
\end{equation*}
Furthermore, for any $n \geqslant 1$, $U^n = \ann (I^n)$.
\end{example}

\begin{remark} \normalfont \label{example}
Let $V$ be an $\FI$-module over an arbitrary commutative ring. By \cite[Theorem 4.8]{CE}, if both $\gd(V)$ and $\hd_1(V)$ are finite, then $V_{\sin}$ is only supported on finitely many objects; that is, $(V_{\sin})_s = 0$ for $s \gg 0$. In particular, $KV$ is supported on finitely many objects. Moreover, one has
\begin{equation*}
\max \{s \geqslant 0 \mid (U^1/U^0)_s \neq 0 \} > \max \{s \geqslant 0 \mid (U^2/U^1)_s \neq 0 \} > \ldots
\end{equation*}
if these numbers are nonzero. Consequently, the sequence $0 = U^0 \subseteq U^1 \subseteq U^2 \subseteq \ldots$ stabilizes after finitely many steps. For details, please refer to \cite{CE, L2}, in particular \cite[Remark 1.1]{LY}.

However, for an $\OI$-module $V$, the kernel $KV$ may be supported on infinitely many objects. For example, let $I$ be the ideal of $\tC$ defined in the previous example, and let $V = M(1) / I M(1)$. The reader can check that for $s \geqslant 0$, $V_s \cong k$ as a $k$-module. Moreover, the natural map $V \to SV$ is a zero map. Therefore, $KV = V$ is supported on infinitely many objects. This example also tells us that in general $S$ and $D$ do not commute, another big difference between $\FI$-modules and $\OI$-modules (for $\FI$-modules, one has $SD \cong DS$, see \cite[Lemma 2.7]{LY}). Indeed, since $KV = V$, one has $DV = SV \cong M(0)$, so $DSV = 0$. However, $SDV = SSV \cong S M(0) \cong M(0)$.

The reason for these differences between behaviors of $\OI$-modules and $\FI$-modules is the \textit{transitivity}, a term appeared in \cite[Subsection 3.1]{GL0}. That is, for $\C = \FI$, the automorphism group $\C(s, s)$ acts transitively from the left sid on $\C (r, s)$ for $s \geqslant r \geqslant 0$. However, for $\C = \OI$, every automorphism group is trivial, so the transitivity fails.
\end{remark}

One has the following simple criterion for stabilization of the recursive procedure after finitely many steps.

\begin{lemma} \label{finite filtration}
Let $V$ be a $\tC$-module, and let $U^0 \subset U^1 \subset U^2 \subset \cdots$ be any increasing chain of $\tC$-submodules of $V$. Let $U=\bigcup_{n=0}^{\infty} U^n$. If $U$ is finitely generated as a $\tC$-module, then there exists a finite number $N \geqslant 0$ such that for each $n \geqslant N$, one has $U^n= U$.
\end{lemma}

We omit the proof of the preceding lemma since it is standard and well known.

\section{Proofs of the main results}

Now we are ready to prove the main theorems of this paper. Firstly we prove Theorem \ref{noetherianity} because it guarantees Noetherianity, and hence we can do homological computations in the abelian category $\tC \module$.

Recall that $\tC \lfMod$ is the category of locally finite $\tC$-modules. If $k$ is Noetherian, this is an abelian category, and $V \in \tC \lfMod$ is finitely generated if and only if $\gd(V) < \infty$. Moreover, the following statements are equivalent:
\begin{itemize}
\item Every finitely generated $\tC$-module $V$ is finitely presented.
\item For $V \in \tC\lfMod$, if $\gd(V)$ is finite, so is $\hd_1(V)$.
\item Every finitely generated $\tC$-module is super finitely presented.
\item For $V \in \tC\lfMod$, if $\gd(V)$ is finite, so is $\hd_i(V)$ for $i \geqslant 0$.
\item The category $\tC \module$ is abelian.
\item The category $\tC$ is locally Noetherian.
\end{itemize}

\begin{proof}[Proof of Theorem \ref{noetherianity}]
Let (P) be the finitely presented property. One direction is trivial. For the other direction, we show that if $V$ is finitely generated, then it has (P). This is achieved by carrying out an induction on the generating degree of $V$. If $V = 0$; that is, $\gd(V) = -1$, the conclusion holds trivially.

For a nonzero $V$, consider the short exact sequence $0 \to V_{\sin} \to V \to V_{\reg} \to 0$. Since $V_{\reg}$ is a quotient module of $V$, $\gd(V_{\reg}) \leqslant \gd(V)$, so $\gd(DV_{\reg}) < \gd(V_{\reg}) \leqslant \gd(V)$. By the induction hypothesis, $DV_{\reg}$ has (P). By Lemma \ref{singular regular decomposition}, the natural map $V_{\reg} \to SV_{\reg}$ is injective. However, since (P) is $D$-predominant by Corollary \ref{n finitely presented}, we deduce that $V_{\reg}$  has (P) as well.

Turning back to the exact sequence $0 \to V_{\sin} \to V \to V_{\reg} \to 0$, from the long exact sequence of homology groups associated with it, we deduce that
\begin{equation*}
\gd(V_{\sin}) \leqslant \max \{\gd(V), \, \hd_1(V_{\reg}) \} < \infty,
\end{equation*}
so $V_{\sin}$ is finitely generated. Recall that, by definition, $V_{\sin} = \bigcup_n U^n$. By Lemma \ref{finite filtration}, we deduce that for some $n\geqslant 0$, one has $U^n = V_{\sin}$, and hence $V/U^n = V_{\reg}$.

We finish the proof by showing that all $V/U^i$, $0 \leqslant i \leqslant n$, satisfy (P) recursively. We have proved the conclusion for $V/U^n = V_{\reg}$. From the exact sequence
\begin{equation*}
0 \to U^n/U^{n-1} \to V/{U^{n-1}} \to S( V/U^{n-1} )  \to D(V/U^{n-1}) \to 0,
\end{equation*}
we get a short exact sequence
\begin{equation*}
0 \to V/U^n \to S( V/U^{n-1} )  \to D(V/U^{n-1}) \to 0.
\end{equation*}
Since $V/U^{n-1}$ is a quotient module of $V$, one has $\gd(D(V/U^{n-1})) < \gd(V/U^{n-1}) \leqslant \gd(V)$. By the induction hypothesis on generating degrees, $D(V/U^{n-1})$ has property (P). We conclude that $S( V/U^{n-1} ) $ also has (P), since $ V/U^n$ has been proved to have (P) (and (P) satisfies condition (C3) by Corollary \ref{n finitely presented}). Finally, $ V/U^{n-1}$ has (P) as well since in Corollary \ref{n finitely presented} we have shown that (P) is $S$-dominant. Repeating this argument for $V/U^{n-1}$, eventually one can prove that all $V/U^i$ has (P), and in particular $V$ satisfies (P). This finishes the proof.
\end{proof}

Using a similar argument, we can prove Theorem \ref{main theorem}.

\begin{proof}[Proof of Theorem \ref{main theorem}]
One direction is trivial. For the other direction, the reader can see that to prove the previous theorem we \textbf{only} used conditions (C1), (C2), (C3), and that $U^n=V_{\sin}$ for some $n\geqslant 0$.

If $S M(s)$ is finitely presented for all $s \geqslant 0$, then $\tC \module$ is abelian by Theorem \ref{noetherianity}. Therefore, for every $V \in \tC \module$, $V_{\sin}$ is finitely generated. By Lemma \ref{finite filtration}, we indeed have $U^n=V_{\sin}$ for some $n\geqslant 0$.
\end{proof}

\begin{remark} \normalfont \label{stability question}
From the proofs of the above theorems we see that the stability of the sequence $U^0\subset U^1 \subset U^2\subset \cdots $ is crucial. One may want to extend the conclusion of Lemma \ref{finite filtration} for arbitrary $\tC$-modules. Explicitly, we want to find a sufficient condition such that for every $\tC$-module $V$ (over general commutative rings) satisfying it, the above sequence $U^n$'s stabilizes after finitely many steps, and obtain quantitative upper bounds for minimal lengths for the stabilization. A potential solution is to find a suitable numerical invariant $\mu(U^n) \in \mathbb{N}$ for each term $U^i$ in the above sequence such that $\mu(U^n) > \mu (U^{n+1})$. However, at this moment we do not have a satisfactory answer since the most important invariant, generating degree, does not satisfy this requirement.

But for some special examples, we do find some perfect numerical invariants which can be used to conclude that the above sequence stabilizes after finitely many steps. For instance, as pointed out in Remark \ref{example}, for an $\FI$-module $V$ over any commutative ring, if both $\gd(V)$ and $\hd_1(V)$ are finite, then the \emph{torsion degree} $\td(V)$ defined in \cite{L2} is finite. Moreover, one has $\td(U^n) > \td(U^{n+1})$. Consequently, after at most $\gd(V) + \hd_1(V)$ steps, the sequence stabilizes. For details, see \cite{CE, L3}.
\end{remark}

The conclusion of Theorem \ref{noetherianity} is that every finitely generated $\tC$-module over a commutative Noetherian ring has property (P). Sometimes people want to show that a property (Q) of $\tC$-modules (which might not be finitely generated) implies another property (P). That is, the class of $\tC$-modules satisfying (Q) is a subclass of $\tC$-modules satisfying (P). With a little modification, we can prove the following result for $\tC \Module$, the category of all $\tC$-modules over an \textbf{arbitrary} commutative ring $k$:

\begin{theorem} \label{main theorem two}
Suppose that $k$ is a commutative ring and $\tC$ has a generic shift functor $S$. Let (P) and (Q) be two properties of some $\tC$-modules and suppose that the zero module has (P) and (Q). Suppose moreover that the following conditions are satisfied:
\begin{enumerate}
\item (P) is $S$-dominant and $D$-predominant in $\tC \Module$ and satisfies condition (C3);
\item (Q) is $S$-invariant and $D$-invariant in $\tC \Module$;
\item if two terms in a short exact sequence of $\tC$-modules satisfy (Q), so does the third term;
\item the sequence $U^0\subset U^1\subset U^2\subset \cdots$ stabilizes after finitely many steps if $V$ satisfies (Q).
\end{enumerate}
Then every $\tC$-module $V$ satisfying (Q) with $\gd(V) < \infty$ also satisfies (P).
\end{theorem}

\begin{proof}
By carefully analyzing the proofs of Theorems \ref{noetherianity} and \ref{main theorem}, we know that to carry out the induction, it suffices to show that all terms $V/U^i$ have property (Q). But this is easy to check. Indeed, consider the short exact sequence $0 \to V/U^1 \to SV \to DV \to 0$. By the given conditions, both $SV$ and $DV$ have (Q), so does $V/U^1$. Eventually one can show that all $V/U^n$ have (Q). Note that in the proof we only require that the sequence $U^0\subset U^1\subset U^2\subset \cdots$ stabilizes after finitely many steps, so the situation that $V$ is infinitely generated or $k$ is not Noetherian is allowed.
\end{proof}

\begin{remark} \normalfont
Let $k$ be a commutative Noetherian ring, and let (Q) be the finitely generated property. Then under the extra assumption that every $S M(s)$ is finitely presented, one knows that $\tC \module$ is abelian, so conditions (2) and (3) in the above theorem hold. Moreover, condition (4) follows from Lemma \ref{finite filtration}. Therefore, we may deduce Theorem \ref{main theorem} from Theorem \ref{main theorem two}.
\end{remark}

We provide several applications of these theorems.

\begin{corollary}
Let $V$ be a $\tC$-module such that $V_n = 0$ for $n \gg 0$. If $\hd_i(SM(s)) < \infty$ for all $i, s \geqslant 0$, then $\hd_i(V) < \infty$ for all $i \geqslant 0$.
\end{corollary}

\begin{proof}
Let (Q) be the following property on $\tC$-modules $V$: $V_n = 0$ for $n \gg 0$. Clearly, this implies the finite generating degree of $V$. Let (P) be the property that $\hd_i(V) < \infty$ for all $i \geqslant 0$. The reader can check that (Q) satisfies conditions (2) and (3) in Theorem \ref{main theorem two}. Furthermore, by Lemma \ref{main technical lemma}, since $S M(s)$ satisfies (P) for every $s \geqslant 0$, (P) is $S$-dominant and $D$-predominant. Consequently, condition (1) in Theorem \ref{main theorem two} is fulfilled as well. To check condition (4), we note that
\begin{equation*}
\max \{ n \geqslant 0 \mid V_n \neq 0 \} > \max \{ n \geqslant 0 \mid (SV)_n \neq 0 \} \geqslant \max \{ n \geqslant 0 \mid (V/U^1)_n \neq 0 \}
\end{equation*}
since $(V/U^1)$ is a submodule of $SV$. From this observation one deduces that the length of the sequence $0=U^0 \subset U^1 \subseteq \ldots$ is bounded by $\max \{n \geqslant 0 \mid V_n \neq 0\}$, a finite number by the given assumption. Therefore, condition (4) holds, too.
\end{proof}

One more application is Corollary \ref{polynomial growth}. The following proof uses the same argument as that of \cite[Theorem C]{CEFN}.

\begin{proof}[Proof of Corollary \ref{polynomial growth}]
Let (P) be the following property of $\tC$-modules: $H_V$ coincides with a rational polynomial for $s \gg 0$. It is plain to check that (P) satisfies the conditions in Theorem \ref{main theorem}. By that theorem, every finitely generated $\tC$-module has (P).

Now we show that the polynomial associated to $H_V$ has degree at most $\gd(V)$. Recall that for some $n\geqslant 0$, we have
\begin{equation*}
0 = U^0 \subset U^1 \subset \cdots \subset U^n = V_{\sin} \quad \mbox{ and }\quad V/U^n = V_{\reg}.
\end{equation*}
Consider the short exact sequence $0 \to V/U^n \to S(V/U^n) \to D(V/U^n) \to 0$. By the induction hypothesis, the Hilbert function of $D(V/U^n)$ coincides with a polynomial of degree at most $\gd(D(V/U^n))$ for $s \gg 0$. But from the short exact sequence we also know that
\begin{equation*}
H_{D(V/U^n)} (s) = H_{V/U^n}(s+1) - H_{V/U^n} (s).
\end{equation*}
Therefore, for $s \gg 0$, $H_{V/U^n}$ coincides with a polynomial function whose degree is at most
\begin{equation*}
\gd(D(V/U^n)) + 1 = \gd(V/U^n) \leqslant \gd(V/U^{n-1}).
\end{equation*}
Now consider the short exact sequence $0 \to V/U^n \to S(V/U^{n-1}) \to D(V/U^{n-1}) \to 0$. We deduce that the Hilbert function of $S(V/U^{n-1})$ coincides with a polynomial with degree at most $\gd(V/U^{n-1})$ for $s \gg 0$, so does the Hilbert function of $V/U^{n-1}$. Eventually, one proves the conclusion for $V = V/U^0$.
\end{proof}

\section{Castelnuovo-Mumford regularity} \label{CM regularity section}

In this section we consider another important application: the Castelnuovo-Mumford regularity of $\tC$-modules. For abbreviation let (P) be the property of having finite Castelnuovo-Mumford regularity. We want to obtain a sufficient criterion such that every finitely generated $\tC$-module satisfies (P). By Theorem \ref{main theorem}, the sufficient criterion must guarantee that (P) is $S$-dominant and $D$-predominant, while the condition (C3) is fulfilled automatically.

Intuitively, a $\tC$-module $V$ has (P) if and only if it eventually has ``Koszul" behavior; that is, its homological degrees is bounded by a linear function with slope one. Motivated by Theorem \ref{noetherianity}, which deals with the situation that all homological degrees are finite, the reader may believe that we should require all $S M(s)$ to satisfy (P) for $s \geqslant 0$; or explicitly, for every $s \geqslant 0$, $\reg(S M(s)) \leqslant s + N_s$ where $N_s$ is a constant depending on $s$. However, it seems to us that this condition is not strong enough for proving the conclusion. In the following lemma we impose a stronger condition; that is, we require the regularity of all $S M(s)$ to be bounded by $s + N$, where $N$ is a nonnegative integer independent of $s$.

\begin{lemma} \normalfont
Let $V$ be a $\tC$-module and suppose that there is a nonnegative integer $N$ such that $\reg(S M(s)) \leqslant s + N$ for all $s \geqslant 0$. If the  map $\mu_V: V \to SV$ is injective, then
\begin{equation*}
\hd_i (V) \leqslant \reg(DV) + (N+1)i + 1.
\end{equation*}
\end{lemma}

\begin{proof}
The conclusion holds clearly if $\reg(DV) = \infty$ or $V= 0$, so we assume that $\reg(DV) < \infty$ and $V$ is nonzero. Let $P^{\bullet} \to V \to 0$ be an adaptable projective resolution defined in Definition \ref{adaptable projective resolution}. As explained in the proof of Statement (5) of Lemma \ref{main technical lemma}, since the map $V \to SV$ is injective, the adaptable projective resolution $P^{\bullet} \to V \to 0$ induces an exact complex $DP^{\bullet} \to DV \to 0$ which in general is not a projective resolution.

\textbf{Claim}: $\reg(D M(s)) \leqslant \reg(S M(s))$, and hence for a term $P^j$ appearing in $P^{\bullet} \to V \to 0$, one has $\reg(DP^j) \leqslant \reg(SP^j)$. Indeed, from the short exact sequence $0 \to M(s) \to SM(s) \to DM(s) \to 0$ and the long exact sequence
\begin{equation*}
\ldots \to H_1(M(s)) = 0 \to H_1 (SM(s)) \to H_1(DM(s)) \to H_0(M(s)) \to H_0(SM(s)) \to H_0(DM(s)) \to 0
\end{equation*}
associated with it, we deduce that
\begin{equation*}
\gd(D M(s)) \leqslant \gd(S M(s)) \leqslant \reg(S M(s)).
\end{equation*}
and
\begin{align*}
\hd_1(D M(s)) - 1 & \leqslant \max \{ \hd_1 (S M(s)) - 1, \, \gd(M(s)) -1 \}\\
& \leqslant \max \{ \hd_1 (S M(s)) - 1, \, \gd(S M(s)) \} & \text{Statement (3) of Lemma \ref{main technical lemma}}\\
& \leqslant \reg(S M(s)).
\end{align*}
For $i \geqslant 2$, we have $\hd_i(S M(s)) = \hd_i(D M(s))$. The claim follows from these observations.

Now we turn to the projective resolution $P^{\bullet} \to V \to 0$ and use it to study homological degrees of $V$. Recall that $Z^0 = V$, $Z^1$ is the kernel of $P^0 \to V$, and $Z^i$ is the kernel of $P^{i-1} \to P^{i-2}$ for $i \geqslant 2$. We show that
\begin{align*}
\gd(Z^i) & \leqslant \reg(DV) + (N+1)i + 1;\\
\reg(DZ^i) & \leqslant \reg(DV) + (N+1)i.
\end{align*}
The conclusion of this lemma follows immediately from the first inequality.

For $i = 0$, these inequalities hold clearly since $\gd(DV) \leqslant \reg(DV)$ and
\begin{equation*}
\gd(V) = \gd(DV) + 1 \leqslant \reg(DV) + 1.
\end{equation*}
Suppose that they hold for all $j$ with $j \leqslant i$ and consider $j = i + 1$. Note that for the short exact sequence $0 \to Z^{i+1} \to P^i \to Z^i \to 0$, the natural map $Z^i \to SZ^i$ is injective since $Z^i$ is either $V$ or a submodule of $P^{i-1}$ for $i \geqslant 1$, and in the latter case Statement (1) of Lemma \ref{main technical lemma} applies. Therefore, by Statement (4) of Lemma \ref{main technical lemma}, we get a short exact sequence $0 \to DZ^{i+1} \to DP^i \to DZ^i \to 0$, and have
\begin{align*}
\reg(DZ^{i+1}) & \leqslant \max \{ \reg(DP^i), \, \reg(DZ^i) + 1 \}\\
& \leqslant \max \{ \reg(SP^i), \, \reg(DV) + i(N+1) + 1 \} & \text{ by the induction hypothesis} \\
& \leqslant \max \{ \gd(P^i) + N, \, \reg(DV) + i(N+1) + 1 \} & \text{ by the given assumption} \\
& \leqslant \max \{ \gd(Z^i) + N, \, \reg(DV) + i(N+1) + 1 \} & \text{ since $\gd(P^i) = \gd(Z^i)$}\\
& \leqslant \max \{ \reg(DV) + (N+1)i + 1 + N, \, \reg(DV) + i(N+1) + 1 \} & \text{  by the induction hypothesis}\\
& \leqslant \reg(DV) + (N+1)(i+1).
\end{align*}
Therefore,
\begin{equation*}
\gd(Z^{i+1}) \leqslant \gd(DZ^{i+1}) + 1 \leqslant \reg(DZ^{i+1}) + 1 \leqslant \reg(DV) + (N+1)(i+1) + 1.
\end{equation*}
The conclusion follows by induction.
\end{proof}

\begin{remark} \normalfont
The intuition underlying the conclusion of this lemma is as follows. If $DV$ has finite regularity, then $\hd_{i+1}(DV) - \hd_i(DV) \leqslant 1$ for $i \gg 0$. However, since in the exact sequence $0 \to DZ^{i+1} \to DP^i \to DZ^i \to 0$ the middle term is not projective, one cannot deduce that $\gd(DZ^{i+1}) \leqslant \gd(DZ^i) + 1$ and hence $\gd(Z^{i+1}) \leqslant \gd(Z^i) + 1$ for $i \gg 0$. Loosely speaking, the nonnegative integer $N$ amplifies the difference between $\gd(Z^{i+1})$ and $\gd(Z^i)$.
\end{remark}

In the case $N = 0$, we have an immediate corollary:

\begin{corollary} \label{D-dominant of regularity}
Let $V$ be a $\tC$-module. Suppose that $\reg(S M(s)) \leqslant s$ for $s \geqslant 0$. If the  map $\mu_V: V \to SV$ is injective, then $\reg(V) \leqslant \reg(DV) + 1$. In particular, the property (P) of having finite Castelnuovo-Mumford regularity is $D$-predominant.
\end{corollary}

The next lemma and its corollary assert that if $N = 0$ in the previous lemma, then the property (P) given in Corollary \ref{D-dominant of regularity} is both $S$-dominant and $S$-invariant.

\begin{lemma}
Suppose that $\reg(S M(s)) \leqslant s$ for all $s \geqslant 0$. Then for any $\tC$-module $V$, one has
\begin{align*}
\hd_i(SV) & \leqslant \max \{ \hd_j(V) + i -j \mid 0 \leqslant j \leqslant i \};\\
\hd_i(V) & \leqslant \max \{ \{ \hd_j(V) + i - j \mid 0 \leqslant j \leqslant i-1  \} \cup \{ \hd_i(SV) + 1 \} \}.
\end{align*}
\end{lemma}

\begin{proof}
Take a short exact sequence $0 \to W \to P \to V \to 0$ such that $P$ is projective and $\gd(P) = \gd(V)$. Applying $S$ we get another short exact sequence $0 \to SW \to SP \to SV \to 0$. We will use these two sequences to prove the inequalities by induction. They are true for $i = 0$ by Statement (2) of Lemma \ref{main technical lemma}. Suppose that they hold for $j$ with $j \leqslant i$, and consider the $(i+1)$-th homological degrees.

\textit{Proof of the first inequality.} By the short exact sequence $0 \to SW \to SP \to SV \to 0$ we have:
\begin{align*}
\hd_{i+1}(SV)  & \leqslant \max \{ \hd_i(SW), \, \hd_{i+1}(SP) \} \\
& \leqslant \max \{\hd_i(SW), \, \gd(P) + i + 1 \} & \text{ since $\reg(SP) \leqslant \gd(P)$}\\
& \leqslant \max \{ \{ \hd_j(W) + i -j \mid 0 \leqslant j \leqslant i \} \cup \{ \gd(P) + i + 1 \} \} & \text{ the induction hypothesis on $W$}
\end{align*}
Note that $\gd(P) = \gd(V)$, $\hd_j(W) = \hd_{j+1}(V)$ for $j \geqslant 1$, and
\begin{equation*}
\gd(W) \leqslant \max \{\gd(P), \, \hd_1(V) \}.
\end{equation*}
Putting them together, we obtain
\begin{align*}
\hd_{i+1}(SV)  & \leqslant \max \{ \gd(V) + i, \, \hd_1(V) + i, \, \hd_2(V) + s - 1, \, \ldots, \, \hd_{i+1}(V), \,  \gd(V) + i + 1\} \\
& = \max \{ \hd_j(V) + i + 1 -j \mid 0 \leqslant j \leqslant i + 1 \}.
\end{align*}
Therefore, the first inequality holds for $i+1$.

\textit{Proof of the second inequality.} Consider the short exact sequence $0 \to W \to P \to V \to 0$. By induction on $i$,
\begin{equation*}
\hd_{i+1}(V) \leqslant \hd_i(W) \leqslant \max \{ \gd(W) + i, \, \ldots, \, \hd_{i-1}(W) + 1, \, \hd_i(SW) + 1 \}.
\end{equation*}
We also have $\gd(W) \leqslant \max \{\gd(V), \, \hd_1(V) \}$ and $\hd_j(W) = \hd_{j+1}(V)$ for $j \geqslant 1$. With these observations, we obtain:
\begin{equation*}
\hd_{i+1}(V) \leqslant \max \{ \gd(V) + i, \, \hd_1(V) + i, \, \ldots, \, \hd_i(V) + 1, \, \hd_i(SW) + 1 \}.
\end{equation*}
However, from the short exact sequence $0 \to SW \to SP \to SV \to 0$ we deduce that
\begin{equation*}
\hd_i(SW) \leqslant \max \{ \hd_i(SP), \, \hd_{i+1}(SV) \} \leqslant \max \{ \gd(P) + i, \, \hd_{i+1}(SV) \}.
\end{equation*}
Combining these two inequalities, we conclude that
\begin{equation*}
\hd_{i+1}(V) \leqslant \max \{ \gd(V) + i + 1, \, \hd_1(V) + i, \, \ldots, \, \hd_i(V) + 1, \, \hd_{i+1}(SV) + 1 \},
\end{equation*}
so the second inequality holds for $i+1$.
\end{proof}

\begin{remark} \normalfont
Note in the above proof we do not require those homological degrees to be finite.
\end{remark}

We describe two useful corollaries.

\begin{corollary} \label{shift equivalent of regularity}
Suppose that $\reg(S M(s)) \leqslant s$ for $s \geqslant 0$ and let $V$ be a $\tC$-module. Then
\begin{equation*}
\reg(SV) \leqslant \reg(V) \leqslant \reg(SV) + 1.
\end{equation*}
In particular, the property (P) of having finite Castelnuovo-Mumford regularity is $S$-dominant and $S$-invariant.
\end{corollary}

\begin{proof}
By the previous lemma, for $i \geqslant 0$, one has
\begin{equation*}
\hd_i(SV) - i \leqslant \max \{ \hd_j(V) -j \mid 0 \leqslant j \leqslant i \} \leqslant \reg(V),
\end{equation*}
so $\reg(SV) \leqslant \reg(V)$.

To show $\reg(V) \leqslant \reg(SV) + 1$, it suffices to check that $\hd_i(V) \leqslant \reg(SV) + i + 1$ for $i \geqslant 0$, which holds for $i = 0$. Suppose that it holds for all nonnegative integers less than or equal to $i$. Now for $i+1$, by the previous lemma,
\begin{equation*}
\hd_{i+1} (V) \leqslant \max \{ \{ \hd_j(V) + i + 1 - j \mid 0 \leqslant j \leqslant i  \} \cup \{ \hd_{i+1}(SV) + 1 \} \}.
\end{equation*}
One has $\hd_{i+1}(SV) + 1 \leqslant \reg(SV) + i + 2$; and for $0 \leqslant j \leqslant i$, by the induction hypothesis,
\begin{equation*}
\hd_j(V) + i + 1 - j \leqslant (\reg(SV) + j + 1) + i + 1 - j = \reg(SV) + i + 2.
\end{equation*}
Therefore, $\hd_{i+1}(V) \leqslant \reg(SV) + i + 2$ as claimed. By induction, $\reg(V) \leqslant \reg(SV) + 1$.
\end{proof}

The second corollary generalizes \cite[Theorem 1.8]{GL} and the second statement in \cite[Theorem 1.5]{L1}.

\begin{corollary}
Suppose that $\reg(S M(s)) \leqslant s$ for $s \geqslant 0$. Let $V$ be a $\tC$-module such that $V_n = 0$ for $n > N$. Then $\reg(V) \leqslant N$. In particular, if $\tC (s, s)$ is a semisimple algebra for every $s \geqslant 0$, then the $i$-th homology group $H_i (\tC (s, s))$ is either zero or is concentrated on object $s + i$ for $s, i \geqslant 0$, where $\tC(s, s)$ is regarded as a $\tC$-module concentrated on $s$; in other words, $\tC$ is a Koszul category.
\end{corollary}

\begin{proof}
Without loss of generality one can assume that $V \neq 0$, $N < \infty$, and $V_N \neq 0$. The first statement can be proved by an easy induction on $N$ together with the conclusion of the previous lemma. To prove the second statement, one just notes that projective covers exist in $\tC \Module$ under the given condition. For details, one can refer to \cite[Subsection 5.2]{L1}.
\end{proof}

Now we are ready to prove Corollary \ref{cm regularity}.

\begin{proof}[Proof of Corollary \ref{cm regularity}]
It has been established in Corollaries \ref{D-dominant of regularity} and \ref{shift equivalent of regularity} that the property of having finite Castelnuovo-Mumford regularity is $D$-predominant and $S$-dominant. Since $\reg(SM(s)) \leqslant s$ for $s \geqslant 0$, all homological degrees of $S M(s)$ are finite, so they are finitely presented as $k$ is a Noetherian ring. Now the reader can apply Theorem \ref{main theorem}, noting that condition (C3) holds automatically.
\end{proof}

\end{document}